\documentclass[aihp]{imsart}

\RequirePackage{amsthm,amsmath,amsfonts,amssymb}
\RequirePackage[numbers]{natbib}
\RequirePackage[colorlinks,citecolor=blue,urlcolor=blue]{hyperref}
\RequirePackage{graphicx}
\newsavebox{\auteurbm}
  {\small\slshape%
  \savebox{\auteurbm}{\upshape\sffamily#1}%
  \begin{flushleft}}
  {\\[4pt]\usebox{\auteurbm}
  \end{flushleft}\normalsize\upshape}
  
  \usepackage{mathtools}
  \usepackage{dsfont}
  \usepackage{amsmath}
  \usepackage{amsthm}

  \theoremstyle{definition}
  \newtheorem{defn}{Definition}[section]
  \newtheorem*{defn*}{Definition}
  \theoremstyle{plain}
  \newtheorem{thm}{Theorem}[section]
  \newtheorem{prop}{Proposition}[section]
  \newtheorem*{prop*}{Proposition}
  \newtheorem{lem}{Lemma}[section]
  \newtheorem{cor}{Corollary}[section]
   \newtheorem*{cor*}{Corollary}
  \newtheorem*{theo*}{Theorem}
  \newtheorem*{thm*}{Theorem}
  
  \theoremstyle{remark}
  
  \newtheorem{rem}{Remark}[section]
  \newtheorem{ex}{Example}
  \newtheorem{nota}{Notation}

 \usepackage[toc,page]{appendix}
\usepackage{calrsfs}
\usepackage{tikz-cd}

\newcommand{\p}{\mathbb{P}}

\newcommand{\R}{\mathbb{R}}
\newcommand{\A}{\mathcal{A}}
\newcommand{\B}{\mathcal{B}}
\newcommand{\C}{\mathcal{C}}
\newcommand{\D}{\mathcal{D}}

\newcommand{\U}{\operatorname{\mathcal{U}}}
\newcommand{\V}{\mathcal{V}}
\newcommand{\id}{\operatorname{id}}

\usepackage{pstricks,pst-node,pst-tree}

\newcommand{\ov}{\overline}
\usepackage{bm}
\usepackage{relsize}
\newcommand{\bigsqcap}{\mathlarger{\mathlarger{\mathlarger{\sqcap}}}}

\newcommand\independent{\protect\mathpalette{\protect\independenT}{\perp}}
\def\independenT#1#2{\mathrel{\rlap{$#1#2$}\mkern2mu{#1#2}}}

 \usepackage{xcolor}
 \newcommand\N{\mathbb{N}}

\newcommand{\I}{\mathcal{P}^2(I)}
\newcommand{\Pa}{\mathcal{P}}
\newcommand{\PP}{\mathcal{P}^2}
\newcommand{\g}{\mathcal{F}}

\DeclareMathOperator*{\vect}{\textbf{Gr}}

\newcommand{\bA}{\mathcal{M}}

\usepackage{appendix}

\usepackage{pdfpages} %pour ajouter les pdf des articles.

\usepackage[intoc]{nomencl}
\makenomenclature

\usepackage{pdfpages}

\newcommand{\pp}{\g(I)}
\newcommand{\GG}{U}
\newcommand{\GP}{\GG(I)}

\makeatletter
\newtheorem*{rep@theorem}{\rep@title}
\newcommand{\newreptheorem}[2]{%
\newenvironment{rep#1}[1]{%
 \def\rep@title{#2 \ref{##1}}%
 \begin{rep@theorem}}%
 {\end{rep@theorem}}}
\makeatother

  \theoremstyle{plain}
\newreptheorem{thm}{Theorem}
\newreptheorem{prop}{Proposition}

\makeatletter
\newcommand{\mylabel}[2]{#2\def\@currentlabel{#2}\label{#1}}
\makeatother

\numberwithin{equation}{section}

\usepackage{multiaudience}
\SetNewAudience{long}
\usepackage{cite}

\begin{document}

\begin{frontmatter}
\title{Bayesian/Graphoid intersection property for factorisation spaces.}
%\title{A sample article title with some additional note\thanksref{t1}}
\runtitle{Intersection property for factorisation spaces}
%\thankstext{T1}{A sample additional note to the title.}

\begin{aug}
%%%%%%%%%%%%%%%%%%%%%%%%%%%%%%%%%%%%%%%%%%%%%%
%%Only one address is permitted per author. %%
%%Only division, organization and e-mail is %%
%%included in the address.                  %%
%%Additional information can be included in %%
%%the Acknowledgments section if necessary. %%
%%%%%%%%%%%%%%%%%%%%%%%%%%%%%%%%%%%%%%%%%%%%%%
\author[A]{\fnms{Gr\'egoire} \snm{Sergeant-Perthuis}\ead[label=e1]{sergeant@phare.normalesup.org}}%,
%\author[B]{\fnms{Second} \snm{Author}\ead[label=e2,mark]{second@somewhere.com}}
%\and
%\author[B]{\fnms{Third} \snm{Author}\ead[label=e3,mark]{third@somewhere.com}}
%%%%%%%%%%%%%%%%%%%%%%%%%%%%%%%%%%%%%%%%%%%%%%
%% Addresses                                %%
%%%%%%%%%%%%%%%%%%%%%%%%%%%%%%%%%%%%%%%%%%%%%%
\address[A]{IMJ-PRG, Universit\'e de Paris \printead{e1}}

%\address[B]{Department,
%University or Company Name,
%\printead{e2,e3}}
\end{aug}

\begin{abstract}

We remark that Pearl's Graphoid intersection property, also called intersection property in Bayesian networks, is a particular case of a general intersection property, in the sense of intersection of coverings, for factorisation spaces, also coined as factorisation models, factor graphs or by Lauritzen in his reference book \emph{Graphical Models} as hierarchical model subspaces. A particular case of this intersection property appears in Lauritzen's book as a consequence of the decomposition into interaction subspaces; the novel proof that we give of this result allows us to extend it in the most general setting. It also allows us to give a direct and new proof of the Hammersley-Clifford theorem transposing and reducing it to a corresponding statement for graphs, justifying formally the geometric intuition of independency, and extending it to non finite graphs. This intersection property is the starting point for a generalization of the decomposition into interaction subspaces to collections of vector spaces \citep{GS2}.

\end{abstract}



\end{frontmatter}

\section{Introduction \label{prelim}}

\subsection{Intersection property}

\subsubsection{Intersection property and graphoids}
To describe the structure of dependencies of a set of random variables, as well said by Judea Pearl in Chapter 3 of \citep{Pearl1988}, one can introduce a ternary operator corresponding to the conditional independence:\\

 "The notion of informational relevance is given [...] through the device of conditional independence, which successfully captures our intuition about how dependencies
 should change in response to news facts".\\

For any three random variables with discrete values, we will note $X\independent Y\vert_{P} Z$ the fact that $X$ is independent of $Y$ conditionally to $Z$ (see Section \ref{chapitre-1-section-4} Equation \ref{condindep}); in the previous expression $P$ will be omitted from now on, as in literature.\\

The intersection property in Bayesian networks, as found in \citep{FirstC} (Chapter 2 Proposition 2.12) or \citep{Pearl1988} (Chapter 3 Theorem 1), is the following proposition.

\begin{repprop}{Graphoid_interesect}[Intersection property]
Let $W,X,Y,Z$ be four random variables that take values in a finite set and for which the probability density $P_{W,X,Y,Z}$ is stricly positive, then,

\begin{equation} \label{interbay}
 X\independent Y\vert (Z,W) \text{ and } X\independent W \vert (Z,Y)  \implies X\independent(Y,W)\vert Z
\end{equation}

\end{repprop}

Semi-graphoids and graphoids were introduced to give a formal set of axioms, on $\independent$, for conditional independence (see \citep{Dawid}, \citep{Pearl1988}); in this context Proposition \ref{Graphoid_interesect} is called the intersection axiom.

\begin{defn}[Semi-graphoid, graphoid \citep{Pearl1988}\citep{Lauritzen}]
A semi-graphoid structure on a collection $I=\coprod_J \{X_j\}$ is a ternary relation on subsets of $I$, that we shall note as $X\independent Y|Z$, such that, for any, $X,Y,Z,W$, disjoint subsets of $I$, 

\begin{enumerate}
\item if $X\independent Y |Z$ then $Y\independent X|Z$;
\item if $X\independent Y|Z$ and $U\subseteq X$, then $U\independent Y|Z$;
\item if $X\independent Y|Z$ and $U\subseteq X$, then $X\independent Y|Z \cup U$;
\item if$X\independent Y|Z$ and $X\independent W|Y\cup Z$ then $X\independent W\cup Y |Z$.

\end{enumerate}

It is a graphoid if furthermore it satisfies the intersection axiom. 

\end{defn}

\subsubsection{Factorisation spaces}

Let us suppose that $X,Y,Z$ take values respectively in finite spaces $\Omega_X$, $\Omega_Y$, $\Omega_Z$. The fact that $X$ is independent of $Y$ conditionally to $Z$ can be restated as a factorisation property on $P_{X,Y,Z}$; for simplicity let us assume that $P_{X,Y,Z}$ is sticly positive, then $X\independent Y\vert Z$ if and only if for any $(x,y,z)\in \Omega_X\times \Omega_Y\times \Omega_Z$,

\begin{equation}
P_{X,Y,Z}(x,y,z)= \frac{P_{X,Z}(x,z)P_{Y,Z}(y,z)}{P_Z(z)}
\end{equation}

where $P_{X,Z}$, $P_{Y,Z}$, $P_Z$ are repectively the marginal probabilities of $(X,Z)$, $(Y,Z)$ and $Z$.\\

If one notes $\g_{Y,Z}$ the set of strictly positive functions on $\Omega_X\times\Omega_Y\times \Omega_Z$ that only depend on $(Y,Z)$ and $\g_{X,Z}\g_{Y,Z}$ the set of functions that are the product of a function of $\g_{X,Z}$ and a function of $\g_{Y,Z}$, then the intersection property can be restated as, for any strictly positive proability law $P_{X,Y,Z,W}$,

\begin{equation}\label{p1:intersection-factorisation}
P_{X,Y,Z,W}\in \g_{W,X,Z}\g_{W,Y,Z}\cap \g_{X,Y,Z}\g_{W,Y,Z}\implies P_{X,Y,Z,W}\in \g_{X,Z}\g_{W,Y,Z}
\end{equation}

We are interested in generalizing this result to intersections of factorisation spaces that we will define now.

\begin{defn}[Factorisation space]
Let $I$ be a finite set, let $\A\subseteq \Pa(I)$, where $\Pa(I)$ is the set of subsets of $I$. Let $(E_i,i\in I)$ be a collection of sets, let $E_a=\prod_{i\in a} E_i$ for any $a\in \Pa(I)$; for $x\in E_I$, we will denote $x_a$ its projection onto $E_a$. The factorisation space over $\A$ is defined as follows,

\begin{equation}
\g_\A= \{ f\in \R_{>0}^{E_I}: \quad \exists (f_a\in \R^{E_a}_{>0},a\in \A),\forall x\in E_I \ f=\prod_{a\in \A} f_a(x_a)\}
\end{equation}
\end{defn}

\begin{nota}
From now on we shall note $\Pa\Pa(I)$ as $\PP(I)$.
\end{nota}

One can extend the previous definition to the case where $\A$ is non finite. To do so let us introduce a notation; for any $\A\subseteq \Pa(I)$, let,

\begin{equation}
\hat{\A}=\{a\in \Pa(I): \ \exists b\in \A, \ a\leq b\}
\end{equation}

\begin{nota}
$\hat{\A}$ is called the lower set of $\A$ and the set of lower sets will be denoted as $\U(\Pa(I))$, i.e. $\U(\Pa(I))=\{\hat{\A} \vert \A \subseteq \Pa(I)\}$.
\end{nota}

\begin{defn}[Generalized factorisation spaces]
Let $I$ be any set and let $\A\subseteq\Pa(I)$, any $f\in  \R_{>0}^{E_I}$ is in $\g_\A$ if and only if there is $n\in \N$, a collection $(a_k\in \hat{\A} , k\in[1, n])$ and a collection $(f_k\in \R^{E_{a_k}})$ such that for any $k\in [1,n]$, $\vert a_k\vert<\infty $ and for any $x\in E_I$,

\begin{equation}
f(x)=\prod_{k\in [1,n]} f_k(x_{a_k})
\end{equation}

\end{defn}

In particular,

\begin{equation}
\g_{\A}= \g_{\hat{\A}}
\end{equation}

\subsubsection{Main theorem}

The result we want to emphasize in this document is that an intersection property still holds for factorisation spaces.

\begin{repthm}{infi-intersect}[Intersection property for factorisation spaces]

Let $I$ be any set, let $(E_i,i\in I)$ be any collection of sets. For any family  $(\A_j)_{j\in J}$ of elements of $\PP(I)$,

\begin{equation}\label{chapitre-1:intro:intersection-property-factorisation}
\bigcap \limits_{j\in J} \g_{\hat{\A}_j}=\g_{\bigcap \limits_{j\in J}\hat{\A}_j}.
\end{equation}

\end{repthm}

A particular case of the intersection property appears in Lauritzen's \emph{Graphical Models} \citep{Lauritzen} in Appendix B Proposition B.5 as a consequence of the decomposition into interaction subspaces, that we will introduce in the next subsection. The proof we give of this result holds in a more general setting and is a direct one that does not rely on the decomposition into interaction subspaces. In fact in \citep{GS2} we show the converse statement that Equation \ref{chapitre-1:intro:intersection-property-factorisation} is a structure property that characterizes collections of vector spaces that can be decomposed into direct sums of subspaces, similarly to the decomposition into interaction subspaces, in other words satisfying the intersection property implies that this collection has such decomposition.\\ 

A direct consequence of Theorem \ref{infi-intersect} is that there is a complete lattice morphism between $(\U(\Pa(I)),\subseteq)$ and factorisation spaces. This remark enable us to prove the Hammersley-Clifford Theorem in a direct and novel manner, pushing properties of the graph of dependencies directly on its graphical model, that we will now sketch and allows us to give a generalization of the Hammersley-Clifford theorem.\\

\subsection{Hammersley-Clifford Theorem}

\subsubsection{Graphical models: Markov fields and Gibbs fields}

A graphical model is a way to express the interactions of random variables through the properties of a graph. For example, let $I$ be a finite set and let $(X_i,i\in I)$ be a collection of random variables. Let us associate to each random variable a vertex of an undirected  graph $G=(I,A)$, where $I$ is its set of vertices and $A$ its set of edges; one could say that two random variables are in interaction if their vertices are nearest-neighbours in $G$ and expect that there is a collection $(f_a\in \R^{E_a} ,a\in  A)$ such that,

\begin{equation}
\ln P= \sum_{a\in A} f_a
\end{equation}

This is an example of a Gibbs state with respect to a potential. The adjacent elements of $i\in G$ will be denoted as $\partial i$.

\begin{defn}[Gibbs States]
Let $I$ be a finite set and $(E_i,i\in I)$ be a collection of finite sets, let $\A\in \PP(I)$ and let $\Phi=(\phi_a\in \R^{E_a},a\in \A)$ be a collection of interactions, which we shall call a potentiel; a Gibbs state with respect to a potential $\Phi$ is defined as follows, for any $x \in E$, 

\begin{equation}
P(x)= \frac{e^{\sum_{a\in \A}\phi_a(x_a)}}{\sum_{y\in E_I} e^{\sum_{a\in \A}\phi_a(y_a)}}
\end{equation}

\end{defn}

\begin{rem}
Any probability law on $E$ is a Gibbs state; furthermore if there is a potential $\Phi=(\phi_a\in \R^{E_a},a\in \A)$ such that a probability law $P$ is a Gibbs state with respect to $\Phi$, then $P$ is in the factorisation space over $\A$.
\end{rem}

There is an other way to specify the interactions of the random variables from the properties of a graph. For example on can imagine that if two vertices $v,u$ are connected only through a third vertex $k$, i.e. any path from $v$ to $u$ pass by $k$, this would mean that the corresponding random variables, $X_v$, $X_u$ are dependent only through $X_k$, i.e. that,

\begin{equation}
X_v\independent X_u \vert X_k
\end{equation}

This is a particular case of spatial Markov property for the probability law of the random variables. There are several, a priori, different way to translate conditional connectedness properties of the graph into conditional independence properties, let us define two of such. Let for $a\subseteq I$, $X_a$ denote $(X_i)_{i\in a}=X_{|a}$.

\begin{defn}[Markov properties]
Let $G=(I,A)$ be a finite graph, a stricly positive probability $P_X$ on a finite set $E=\underset{i\in I}{\prod}E_i$ obeys,

\begin{enumerate}
\item $(P)$ the pairwise Markov property relative to $G$, if for any pair $(i,j)$ of non-adjacent vertices 

$$X_i \independent X_j | X_{I\setminus \{i,j\}}.$$

\item $(L)$ the local Markov property relative to $G$, if for any vectex $i\in V$, 

$$ X_i \independent X_{I \setminus (i\cup \partial i) }|X_{\partial i}$$

\end{enumerate}

And we call the respective sets $P(G)$, $L(G)$. 

\end{defn}

As we will see the Hammersley-Clifford theorem asserts that the two points of view for reading the interactions from a graph, the Gibbs state and Markov property point of views, are in fact equivalent for a strictly positve probability law. One of the ways to prove the Hammersley-Clifford theorem is to build a decomposition into interaction subspaces of the factorisation spaces \citep{Speed}, we shall therefore give a brief presentation of this decomposition even though we shall not be using it in the rest of this document.

\subsubsection{The decomposition into interaction subspaces}

Let $I$ be a finite set and let $(E_i,i\in I)$ be a collection of finite sets. Let us consider the canonical scalar product on  $\R^{E_I}$, i.e. for any $f,g\in \R^{E_I}$,  

\begin{equation}
\langle f,g\rangle =\sum_{x\in E_I}f(x)g(x)
\end{equation}

Let for any $\A\in \PP(I)$, 

\begin{equation}
U(\A)= \ln \g_\A
\end{equation}

\begin{thm}[Decomposition into interaction subspaces]\label{chapitre-1:intro:decomposition-into-iteraction-spaces}
There is a collection of vector subspaces of $\R^{E_I}$, $(S_a,a\subseteq \Pa(I))$, such that, for any $a\subseteq \Pa(I)$,

\begin{equation}
U(\{a\})= \bigoplus_{b\subseteq a} S_b
\end{equation}

and any two $S_a,S_b$, with $a\neq b$, are orthogonal to one another.
\end{thm}

Several proofs of this result can be found in \citep{Speed}.\\

\subsubsection{A new proof of the Hammersley-Clifford Theorem}

The Hammersley-Clifford theorem states that any Markov condition for a stricly positive probability law can be restated as a condition on the locality of the interactions of its potential, in other words Markov conditions correspond to some factorisation spaces.\\

Let $G=(I,A)$ be a graph; a clique of $G$ is a subset of $G$ such that every two distinct vertices are adjacent. We will note $\mathcal{C}$ the set of its cliques.\\

\begin{thm}[Hammersley-Clifford]\label{p1:HC}
 
 Let $G=(I,A)$ be a finite graph. For all $P_X$ strictly positive probability law on a finite set $\prod_{i\in I}E_i$,

\begin{equation}
P_X\in P(G) \iff P_X\in L(G)\iff P_X \in  \g_\mathcal{C}.
\end{equation}

\end{thm}

The intesection property for factorisation spaces enables us to bring back the proof of the Hammersley-Clifford theorem to a general property on graphs, let us sketch the proof that will will present in more details in this document.\\

Let $(i,j)$ be a pair of elements of $I$, let $[i,j]=\{I\setminus \{i\}), I\setminus \{j\})\}$ and let,

\begin{equation}
\mathcal{A}_{P}=\underset{(i,j): \text{ } i\notin  \partial j}{\bigcap}\widehat{[i,j]}
\end{equation}

Similarly for all $i\in I$, let $[i]=\{I \setminus i, i\cup \partial i  \}$ and let $\mathcal{A}_{L}= \underset{i}{\bigcap}\hat{[i]}$.\\

By remarking that,
\begin{equation}
\hat{\mathcal{A}}_{L}=\hat{\mathcal{A}}_{P}= \mathcal{C}
\end{equation}

and applying the intersection property for factorisation spaces one ends the proof.

\subsection{Structure of this document}

In Section \ref{chapitre-1:section-1}, we will give some general properties on partial coverings and their natural order making it a preorder with join and meet. Proposition \ref{morph} states that there is an increasing function between the preorder set of partial coverings and the poset of factorisation spaces that preserve the join. \\

In Section \ref{chapitre-1:section-intersection-1} we prove the intersection property (Theorem \ref{chapitre-1:central}) and as a consequence we show that the increasing function also preserves meets. In this section we do not assume the $(E_i, i\in I)$ to be finite, however we assume $I$ to be finite. \\

In the next section, Section \ref{chapitre-1:section-intersection-2}, we extend the intersection property to any sets $I$, Theorem \ref{infi-intersect}.\\

Finaly in Section \ref{chapitre-1:section-applications} we give apply the previous theorems giving new proofs for classical results around factorisation spaces that allow us to extend them, in particular we give a generalization of the Hammersley-Clifford theorem. 
\section{Order on partial coverings and factorisation spaces}\label{chapitre-1:section-1}

In this section $I$ is a finite set, let $E=\prod_{i\in I} E_i$ be a product of any sets. For $a,b\in \Pa(I)$ such that $b\subseteq a$ let $\pi^a_b:E_a\to E_b$ be the projection of $E_a$ onto $E_b$ where by convention $\pi^a_\emptyset:E_a\to \ast$ is the projection on the set with one element $\ast$; for $x\in E$, we shall note $\pi_a(x)$ as $x_a$. In particular $\g_\emptyset$ is the set of stricly positive constant functions and for any $a\in \Pa(I)$ we note $\g_{\{a\}}$ as $\g_a$.\\

$\R_{>0}$ can be seen as a vector space for the product law and the exponentiation and similarly for the product of these spaces. In this section we keep the, unusual, product convention to stay closer to the spirit of factorisation.

\subsection{Order on partial coverings} \label{section-coverings}

Any subset $\A\subseteq \Pa(I)$ can be seen as a partial covering of $I$ of support $\cup_{a\in \A} a$. The order for partial covering that we will now introduce is a direct extension of the usual one for coverings.

\begin{defn}\label{order}
 Let us define an intersection $\sqcap$ and a relation $ \text{R}$ on $ \I$. For all $\A, \B \in \I$,

\begin{equation}
\A \text{ R } \B \quad \iff \quad \forall a\in \A,\quad \exists b\in \B,\quad   a\subseteq b
\end{equation}

\begin{equation}
\A\sqcap \B= \{a\cap b \quad | a\in \A, b\in \B \}
\end{equation}
\end{defn}

\begin{prop}
$\text{R}$ is pre-order that we will note $\leq$ and for $\A,\B,\C, \D\in \I$,
\begin{equation}\label{property1}
\A\sqcap \B = \B\sqcap \A, \quad (\A\cup\B)\sqcap \C= (\A\sqcap \C)\cup (\B\sqcap \C),\quad \A\sqcap \B\leq \A \quad .
\end{equation}
\begin{equation}\label{property2}
[\A\leq \C \wedge \B \leq \D ] \implies \A\cup \B \leq \C\cup \D.
\end{equation}
\begin{equation}\label{property3}
[\A\leq \C \wedge \B \leq \D ] \implies \A\sqcap \B \leq \C\sqcap \D .
\end{equation}

where $\wedge$ is the logic operator "and".
\end{prop}

\begin{proof}
 Let $\A,\B,\C \in \I$. For all $a\in \A$, $a\subseteq a$. Therefore $\A\leq  \A$. Assume, $\A \leq \B$ and $\B\leq C$, then,
 
$$\forall a\in \A,\quad \exists b\in \B ,  a\subseteq b \quad \quad \forall b\in \B,\quad \exists c\in \C ,  b\subseteq c \quad .$$

For $a\in \A$ there is $b\in \B $ and $c\in \C$ such that  $a\subseteq b\subseteq c$ so $a\subseteq c$ and $[\A \leq \B \wedge \B\leq \C ] \implies \A \leq \C$. Therefore $\leq$ is a pre-order.

$$\exists a\in \A, \exists b\in \B,\quad  x=a\cap b \iff \exists a\in \B, \exists b\in \A, x=a\cap b$$

So $\A\sqcap \B= \B\sqcap \A$.

$$(\A\cup\B)\sqcap \C=\underset{(a,c)\in \left(\A\cup\B\right) \times \C}{\bigcup} \{a\cap c\}= \underset{\substack{(a,c)\in \A\times \C \\ \ or \ (a,c)\in \B \times \C}}{\bigcup}\{a\cap c\}=
\underset{(a,c)\in \A\times \C }{\bigcup} \{a\cap c\} \quad \cup \underset{(b,c)\in \B\times \C }{\bigcup} \{b\cap c\}  \quad .$$

So $(\A\cup\B)\sqcap \C= (\A\sqcap \C)\cup (\B\sqcap \C)$.\\

Let $c\in \A\sqcap \B$ then there is $a\in A$, $b\in \B$ such that, $c\subseteq a\cap b\subseteq a$. So,
$[\A\sqcap \B\leq \A ] \wedge [\A\sqcap \B \leq \B]$.\\

Assume $\A\leq \C$ and $ \B \leq \D$ then for all $a\in \A$ there is $c\in \C$ such that $a\subseteq c$, for all $b\in \B$ there is $d\in \D$ such that $b\subseteq d$. So for $x\in \A\cup \B$ there is $c\in \C$ such that $x\subseteq c$ or $d\in \D$ such that $x\subseteq d$. However $c$ and $d\in \C\cup \D$ so $\A\cup \B \leq \C\cup \D$. The last is proven the same way noting that $a\subseteq c$, $b\subseteq d$ implies $a\cap b \subseteq c\cap d$.

\end{proof}

\begin{defn}\label{pre-order}
Let us introduce the usual equivalence relation for a pre-order (see E.III.3 \citep{Bourbaki}), for all $\A,\B\in \I$,

\begin{equation}
\A \sim \B \quad \iff \quad [\A\leq \B] \wedge [\B\leq \A].
\end{equation}

Let $q:\I\to J$, with $J$ any poset, be a pre-order morphism, in the sense that for any $a,b\in \I$ such that $a\leq b$, $q(a)\leq q(b)$. $q$ is said to preserve the equivalence relation when for all $\A,\B\in \I$, $\left[\A \sim \B \implies q(\A)=q(\B)\right]$. In what follows we supporse that $q$ preserves the equivalence relation. \\

If, for any $f:\I\to K$, with $K$ a poset, that is a pre-order morphism and that preserves the equivalence relation, there is a unique $\ov{f}$ that is a poset morphism such that $f=\ov{f}\circ q$, then we will say that $q$ verifies the universal property $(P)$.

\end{defn}

Let us note $\I/\sim$ as $\ov{\I}$.\\

\begin{prop}

If two pre-order morphism, $p_1:\I\to J$, $p_2:\I \to K$, that preserve the equivalence relation, verify the universal property $(P)$, then there is a poset isomorphism between $J$ and $K$.\\

Let us define $p$ as,

$$\begin{array}{ccccc}
p& : &\I & \to &\ov{ \I}\\
& & A & \mapsto &[A]\\
\end{array}$$

There is a unique order $\overline{\leq}$ on $\ov{\I}$ such that $p:(\I,\leq)\to (\ov{\I},\overline{\leq})$ is a pre-order morphism and verifies $(P)$. It verifies for all $\A,\B\in \I$,

\begin{equation}
[\A]\ov{\leq} [\B]\iff  \A \leq \B.
\end{equation}
\smallskip

Furthermore one can define a union on $\ov{\I}$ and an intersection such for all $\A$,$\B$,
 
\begin{equation}\label{quotient_union}
[\A \cup \B]=[\A]\cup[\B], \quad [\A \sqcap \B]=[\A]\sqcap[\B] \quad .
\end{equation}

Equations \ref{property1},\ref{property2},\ref{property3} stay true on $\ov{\I}$. Let us recall them, $\A,\B,\C, \D\in \ov{\I}$,

\begin{equation}
\A\sqcap \B = \B\sqcap \A, \quad (\A\cup\B)\sqcap \C= (\A\sqcap \C)\cup (\B\sqcap \C),\quad \A\sqcap \B\leq \A \quad .
\end{equation}
\begin{equation}
[\A\leq \C \wedge \B \leq \D ] \implies \A\cup \B \leq \C\cup \D.
\end{equation}
\begin{equation}
[\A\leq \C \wedge \B \leq \D ] \implies \A\sqcap \B \leq \C\sqcap \D .
\end{equation}

\end{prop}

\begin{proof}
Let $p_1:\I\to J$, $p_2:\I \to K$, that preserve the equivalence relation, verify the universal property $(P)$. Then there is $\ov{p_1}$, $\ov{p_2}$, two poset morphisms, such that $p_1=\ov{p_1}\circ p_2$, $p_2=\ov{p_2}\circ p_1$. So $p_1=\ov{p_1}\circ \ov{p_2} \circ p_1$, in other words the following diagram commutes:

\begin{equation}
\begin{tikzpicture}[baseline=(current  bounding  box.center),node distance=2cm, auto]
\node (A) {$\I$ };
\node (B) [right of=A] {$J$};
\node (C) [below of =B] {$J$};
\draw[->] (A) to node  {$p_1$} (B);
\draw[->] (B) to node {$\ov{p_1}\circ \ov{p_2}$} (C);
\draw[->] (A) to node [left]{$p_1$} (C);

\end{tikzpicture}
\end{equation}

But $p_1=\id \circ p_1$, therefore by the unicity statement in $(P)$, $\ov{p_1}\circ \ov{p_2}=\id$. One also has that $p_2=\ov{p_2}\circ \ov{p_1} \circ p_2$, so $\ov{p_2}\circ \ov{p_1} =\id$. Therefore $\ov{p_1}$ is a poset isomorphism between $J$ and $K$.\\

Le us define the following relation for $x,y\in \ov{\I}$,

\begin{equation}
x\ov{\leq}y \quad \iff \quad \exists \A, \exists \B,\quad x=[\A] \quad \wedge \text{ } y=[\B] \quad \wedge \quad  \A\leq \B \quad . 
\end{equation}

$(\ov{\I}, \ov{\leq})$ is a poset (see E.III.3 \citep{Bourbaki}).\\

Let $f:\I\to K$, with $K$ a poset, be a pre-order morphism that preserves the equivalence relation. By the universal property for the quotient map,there is a unique $\ov{f}$ such that $f=\ov{f}\circ p$. For $\A,\B\in \I$, suppose $[\A]\ov{\leq} [\B]$, then $\A\leq \B$ and $f(\A)\leq f(\B)$. $\ov{f}([\A])=f(\A)$ and $\ov{f}([\B])=f(\B)$, so $\ov{f}([\A])\leq \ov{f}([\B])$. Therefore $\ov{f}$ is a poset morphism.\\

Suppose that there are two orders $\leq_1$ and $\leq_2$ on $\ov{\I}$ such that $p:(\I,\leq)\to (\ov{\I},\leq_1)$ and $p:(\I,\leq)\to (\ov{\I},\leq_2)$ are pre-order morphism and verify $(P)$. Then there is $\ov{p}$, a poset isomorphism, such that $p=\ov{p}\circ p$. But by the universal property for the quotient map, $\ov{p}=\id$. Therefore $\id:(\ov{\I},\leq_1)\to (\ov{\I},\leq_2)$ is a poset isomorphism. For all $x,y\in \ov{\I}$,

 $$x\leq_1 y \iff x\leq_2 y \quad .$$
 
 So $\leq_1=\leq_2$.\\

Let $\A, \B, \C, \D \in \I$, such that $\A\sim \C$, $\B \sim \D$, then by property Eq \ref{property2}, $\A\cup \B \leq \C\cup \D$ and $ \C\cup \D \leq \A\cup \B$, so $\A\cup \B \sim \C\cup \D $.\\

Similarly, by property Eq \ref{property3} $\A\sqcap \B \leq \C\sqcap \D$ and $ \C\sqcap \D\leq \A\sqcap \B $, so $\A\sqcap \B\sim \C\sqcap \D$. Therefore the union and intersection given by Eq \ref{quotient_union} are well defined. \\

For any $\A,\B,\C, \D\in \I$,

$$[\A]\sqcap [\B]=[\A\sqcap \B] =[\B]\sqcap [\A]= [\B\sqcap \A].$$

$$ ([\A]\cup[\B])\sqcap [\C]= [(\A\cup\B)\sqcap \C]=[(\A\sqcap \C)\cup (\B\sqcap \C)]=([\A]\sqcap [\C])\cup ([\B]\sqcap [\C]).$$

$$[\A\sqcap \B]\leq [\A].$$

Therefore, $[\A]\sqcap [\B]\leq [\A]$. And one proceeds similarly for the two other properties. 

\end{proof}

We will now also note $\ov{\leq}$ as $\leq$.\\

\begin{ex} Consider $I=\{1,2,3,4\}$. $\{\{1, 2\},\{1,3\}\}\leq \{I\}$ and this is true for any element of $\I$.

$$\{\{1, 2\},\{1,3\}\}\cup \{\{2\}\}=\{\{1, 2\},\{1,3\},\{2\}\}\sim \{\{1, 2\},\{1,3\}\}.$$

$$\{\{1, 2, 4\},\{1,3\}\} \sqcap \{\{2,4\},\{2,3\}\}=\{ \{2,4\},\{2\},\emptyset, \{3\}\}\sim \{ \{2,4\},\{3\}\}.$$

\end{ex}

\begin{rem}\label{rem1} By construction, any section of $p$ induces a poset isomorphism. For example the application that sends $\A$ to its lower set induces a section $s:[\A]\mapsto \hat{\A}$ of $p$; $\widehat{\I}$ and $p_{|\U(\Pa(I))}$ is a poset isomorphism. On $\U(\Pa(I))$, $\leq$ is equal to the inclusion $\subseteq$ and $\sqcap=\cap$.
\end{rem}

\subsection{Increasing function from $\PP(I)$ to the poset of factorisation spaces}

Let us denote $\pp$ the poset of factorization spaces.

\begin{prop}\label{morph}
Let,

$$\begin{array}{ccccc}
\Phi& : &\I & \to & \pp\\
& & \A & \mapsto &\g_\A\\
\end{array}$$
$$\begin{array}{ccccc}
\ov{\Phi}& : &\ov{\I} & \to & \pp\\
& & [\A] & \mapsto &\g_\A\\
\end{array}$$

$\ov{\Phi}: (\ov{\I},\leq)\to (\pp,\subseteq)$ is a poset morphism. For all $\A,\B\in \I, \Phi(\A\cup \B)=\Phi(\A).\Phi(\B) $, $\ov{\Phi}([\A]\cup [\B])=\ov{\Phi}([\A]).\ov{\Phi}([\B]) $.\\

If for all $i\in I$, $|E_i|\geq 2$ then $\ov{\Phi}$ is injective and is a poset isomorphism.

\end{prop}

Let us remark that for all $a,b\subseteq I$ such that $a\subseteq b$, $\g_a\subseteq \g_b$ and that for all $a\in \A$, $\g_a\subseteq \g_\A$.\\

Indeed, $\pi_a=\pi^b_{a}\circ \pi_b$ so for all $f:E_a\to \R_{>0}$, $f\circ \pi_a=(f\circ\pi^b_{a})\circ \pi_b$, so $f\circ \pi_a \in \g_b$. Let us note $1$ the constant function equal to $1$. For all $a\subseteq I $, $1\in \g_\emptyset\subseteq \g_a$. For $a\in \A$, $f\in \g_a$, $f=f \underset{b\in \A \setminus \{a\}}{\prod}1$, so $f\in \g_\A$. \\

Let us now prove Proposition \ref{morph}.
\begin{proof}
Let $\A,\B \in \I$ such that $\A\leq \B$ and $f\in \g_\A$ such that $f= \underset{a\in \A}{\prod}g_a $. For all $a \in \A$ there is $b(a)\in \B$ such that $a \subseteq b(a)$, so $g_a\in \g_{b(a)}\subseteq \g_\B$ and $\underset{a\in \A} {\prod}g_a\in \g_\B$ as $\g_\B$ is a vector space. So $\g_\A\leq \g_\B$.\\

Let $\A\leq \B$ and $\B\leq \A$ then $\g_\A\subseteq \g_\B$ and $\g_\B\subseteq \g_\A$, then $\g_\A=\g_\B$ and $\ov{\Phi}$ is well defined and is a poset morphism.\\

For all $\A,\B\in \I$, $\Phi(\A)$ and $\Phi(\B)$ are subspaces of $\Phi(\A\cup \B)$ so $\Phi(\A).\Phi(\B)\subseteq \Phi(\A\cup \B)$. For all $a\in \A\cup \B$, $\g_a \subseteq \Phi(\A).\Phi(\B)$; $\Phi(\A).\Phi(\B)$ also being a vector space, $\Phi(\A\cup \B)\subseteq \Phi(\A).\Phi(\B)$.\\

If for all $i\in I$, $|E_i|\geq 2$ , Corollary 2 in \citep{Yeung} stipulates that $\g_\A=\g_\B$ if and only if $\hat{\A}=\hat{\B}$ but the proof of this results shows that if $\g_\A\subseteq \g_\B$ then $\hat{\A}\leq \hat{\B}$. So $\Phi_{|\hat{\I}}$ is injective therefore so is $\ov{\Phi}$ by remark \ref{rem1}. Furthermore $\ov{\Phi}([\A])\subseteq \ov{\Phi}([\B])$ implies $[\A]\leq [\B]$, so $\ov{\Phi}$ is a poset isomorphism.

\end{proof}

\begin{rem}\label{remark} Proposition \ref{morph} is a very general property for any increasing function $\Gamma$ from any poset $(\A,\leq)$ to $\vect V$ the set of vector subspaces of a vector space $V$. Indeed let $\U,\V\in \Pa(\A)$, $\sum\limits_{a\in \U} \Gamma(a) +\sum\limits_{b\in \V} \Gamma(b)= \sum\limits_{a\in \U\cup \V} \Gamma(a)$, and if $\U\leq \V$, in the same sense than in Definition \ref{pre-order}, then $ \sum\limits_{a\in \U} \Gamma(a) \subseteq  \sum\limits_{a\in \V}\Gamma(a)$. We enounced it as a proposition in order to clarify the presentation, as we use it as a know fact in later proofs.
\end{rem}

\section{ Intersection property for factorisations on finite posets} \label{chapitre-1:section-intersection-1}

In this section we still assume that $I$ is finite. For $a,b,c \subseteq I$ such that $b\cup c =a$ and $b\cap c =\emptyset$, the map $\pi^a_{(c,b)}:E_a \to E_b\times E_c$ is a bijection. We will note for $u\in E_b, v\in  E_c$, ${\pi^{a}_{(c,d)}}^{-1}(u,v)$ as $uv$, in particular for $x\in E$, $x_a= \pi^{a}_{b}(x_a)\pi^{a}_{c}(x_a)=x_b x_c$. Thus we can also write, for any $a,b\subseteq I$, $x_a = x_{a\cap b}x_{a\cap \overline{b}}$.

\begin{lem} \label{decomp}

Let $a\subseteq I$, $\B\in \I$,

\begin{equation}
\g_a \cap \g_\B \subseteq \g_{\{a\}\sqcap \B}.
\end{equation}
\end{lem}
\begin{proof}

Let $f\in \g_a$ and $(g_b)_{b\in \B}\in \underset{b\in \B}{\prod}\g_b$ such that for all $x\in E$,

\begin{equation}
f(x)= \underset{b\in \B}{\prod}g_b(x)\quad .
\end{equation}

There are $f_a, (\tilde{g}_b)_{b\in \B}$ such that for all $x\in E$, $b\in \B$, $f(x)=f_a(x_a)$, $g_b(x)=\tilde{g}_b(x)$.\\

For all $x\in E$,

$$f_{a}(x_{a})=\prod \limits_{b\in B} \tilde{g}_b(x_{b\cap a}x_{b\cap \overline{a}})$$

Let $c_{\overline{a}}\in E_{\ov{a}}$ then, $\pi_a(x_a c_{\ov{a}})=x_a$ and $\pi_{b}(x_a c_{\ov{a}})= (x_{b\cap a}c_{b\cap \ov{a}})$. So,

$$f_a(x_a)=\prod \limits_{b\in B} \tilde{g}_b(x_{b\cap a} c_{b\cap \overline{a}})$$

Let us pose for all $b\in \B$, $g_{1,b}(x_{b\cap a})=  \tilde{g}_b(x_{b\cap a} c_{b\cap \overline{a}})$, then $f=\prod \limits_{b\in B} g_{1,b}\circ\pi_{b\cap a}$.

\end{proof}

\begin{thm}[Intersection property]\label{chapitre-1:central}

Let $I$ be a finite set and let $(E_i)_{i\in I}$ be family of non necessarily finite sets.\\

For $\A,\B\in \I$,  $f\in \R_{>0}^E$, $(f_a)_{a\in \A}\in \underset{a\in \A}{\prod}\R_{>0}^{E_a}$ and $(g_b)_{b\in \B}\in \underset{b\in \B}{\prod}\R_{>0}^{E_b}$ such that, for all $x\in E$,

$$f(x)=\underset{a\in \A}{\prod}f_a(x_a)= \underset{b\in \B}{\prod}g_b(x_b)\quad .$$

There is $(h_{a,b})_{(a,b)\in \A\times \B}\in \underset{(a,b)\in \A\times \B}{\prod}\R_{>0}^{E_{a\cap b}}$ such that for all $x\in E$,
$$f(x)=\underset{(a,b)\in \A\times \B}{\prod}h_{a,b}(x_{a\cap b}).$$

Equivalently,

\begin{equation}\label{newintersecteq}
\g_\A \cap \g_\B\subseteq \g_{\A\sqcap \B}\quad . 
\end{equation}

\end{thm}

\begin{proof}

For $\A, \B \in \I$, $\A\sqcap \B\leq \A$, $\A\sqcap \B\leq \B$. Therefore by Proposition.(\ref{morph}) $\g_{\A\sqcap \B}\subseteq \g_\A \cap \g_\B$.\\

Let us prove the other inclusion by induction on $|\A|$. \\

$|\A|=1$ is the previous Lemma.\ref{decomp}.\\

Suppose that for all $\A,\B\in \I$ such that $|\A|=n\in \N$, $\g_\A \cap \g_\B\subseteq \g_{\A\sqcap \B}$.\\

Let $\A\in \I $, $|\A|=n+1$. Take $\alpha\in  \A$, $|\A \setminus \{\alpha\}|=n$. Pose $\C=\A \setminus \{\alpha\}$. Let $f \in \g_\A \cap \g_\B$, then there is $h_1\in \g_{\alpha}$, $f_1\in \g_\C$, $g\in \g_\B$ such that $f=h_1.f_1=g \quad$.\\

So $h_1=\frac{g}{f_1}$ and  $h_1\in \g_{\C}.\g_\B$. So by Proposition.(\ref{morph}), $h_1\in \g_{\C\cup \B}$. Then by Lemma.\ref{decomp} $h_1\in \g_{(\C\cup \B)\sqcap \{\alpha\}}$. But $(\C\cup \B)\sqcap \{\alpha\}= (\C\sqcap \{\alpha\})\cup (\B\sqcap \{\alpha\})$.\\
 
 So $h_1\in \g_{\C \sqcap \{\alpha\}} . \g_{\B\sqcap \{\alpha\}}$. Furthermore $f_1\in  \g_{\C}$ so $f=h_1.f_1 \in   \g_{\C}. \g_{\C \sqcap \{\alpha\}} . \g_{\B\sqcap \{\alpha\}}$. But $\g_{\C \sqcap \{\alpha\}}\subseteq  \g_{\C}$ so $ \g_{\C}. \g_{\C \sqcap \{\alpha\}}\subseteq  \g_{\C}$ (it is even equal).\\

So there is $f_2\in \g_{\C}$, $h_2\in  \g_{\B\sqcap \{\alpha\}}$ such that $g=h_2.f_2$. Therefore $f_2= \frac{g}{h_2}$. But $\g_{\B\sqcap \{\alpha\}}\subseteq \g_\B$ so $f_2\in \g_\B$.\\

Therefore by the induction hypothesis, $f_2\in \g_{\C \sqcap \B}$, and so $f\in \g_{\B\sqcap \{\alpha\}}\g_{\C \sqcap \B}$. One remarks that $( \{\alpha\}\sqcap \B)\cup (\C \sqcap \B)=\A\sqcap \B$ so $f\in \g_{\A\sqcap \B}$. Which ends the proof by induction.\\
\end{proof}

\begin{cor}\label{Morphisminter} For all $\A,\B\in \I$,

\begin{equation}\label{newintersect}
\g_\A \cap \g_\B= \g_{\A\sqcap \B}=\g_{\hat{\A}\cap \hat{\B}} \quad.
\end{equation}

Which can be rewritten as, for all $\A,\B \in \ov{\I}$,

\begin{equation}
 \ov{\Phi}(\A\sqcap \B)=\ov{\Phi}(\A) \cap \ov{\Phi}(\B) \quad .
\end{equation}

\end{cor}
\begin{proof}
$\A\sqcap \B \leq \A$ and $\A\sqcap \B \leq \B$ therefore $\g_{\A\sqcap \B}\subseteq \g_\A$ and $\g_{\A\sqcap \B}\subseteq \g_\B$.\\

\end{proof}

\section{Extension for infinite posets} \label{chapitre-1:section-intersection-2}

In this section $I$ is any set; let us now use the summation convention instead of the product one. We would like to give a similar definition of $\g_{\mathcal{A}}$ but for infinite posets. If for any $\mathcal{A}\subseteq \mathcal{P}(I)$, we defined $U(\mathcal{A})$ as $\sum\limits_{\substack{a\in \mathcal{A}\\ |a|<+\infty }} U(a)$ then $U(I)=0$. One needs to consider only lower sets in $\U(\Pa(I))$.

Let us call $\GG=\GG_{\Pa(I)}$ and $\GP$ the poset constituted of the $\GG_{\A}$; let $\Psi$ be such that,

\begin{equation}
\begin{array}{ccccc}
\Psi& : &\U(\Pa(I)) & \to & \GP\\
& & \A & \mapsto &\GG_\A\\
\end{array}
\end{equation}

In particular, 

\begin{equation}
\GG(\A)= \ln \g_{\A}
\end{equation}

\begin{cor}\label{inf-intersect} For all $\A,\B\in \U(\Pa(I))$,

\begin{equation}
\GG_\A \cap \GG_\B=\GG_{\A\cap \B}.
\end{equation}

\end{cor}

\begin{proof}
Let $f\in \GG_\A \cap \GG_\B$. There are by definition, $\C_1 \subseteq \A$, $\C_2 \subseteq \B$, that are of finite cardinal, such that $f\in \GG_{\C_1}$ and $f\in \GG_{\C_2}$. By Corollary \ref{Morphisminter}, $f\in \GG_{\C_1\sqcap \C_2}$. As $\C_1\sqcap \C_2\subseteq \A\cap \B$, $f\in \GG_{\A\cap \B}$. 
\end{proof}

We will now show that a stronger version of Corollary.(\ref{inf-intersect}) holds for the intersection on any family of elements of $\U(\Pa(I))$.

\begin{thm}\label{infi-intersect} For any family  $(\A_j)_{j\in J}$ of elements of $\U(\Pa(I))$,

\begin{equation}
\bigcap \limits_{j\in J} \GG_{\A_j}=\GG_{\bigcap \limits_{j\in J}\A_j}.
\end{equation}

\end{thm}

Before giving a proof of this result, let us first state the following lemma, 

\begin{lem}\label{vect-int}
Let $V_1,V_2$ be two vector subspaces of $\GG$. If for any finite $a\in \Pa(I)$, 

\begin{equation}
V_1\cap \GG_a \subseteq V_2\cap \GG_a.
\end{equation}

Then,

\begin{equation}
V_1\subseteq V_2
\end{equation}
\end{lem}

\begin{proof}
Let $v\in V_1$, there is a finite collection of finite subsets of $I$, $(a_k)_{1\leq k\leq n}$, such that, $v\in \underset{1\leq k\leq n}{\sum}\GG_{a_k}$.\\

Therefore $v\in \GG_{(\underset{1\leq k\leq n}{\bigcup} a_k)}$. But $\underset{1\leq k\leq n}{\bigcup} a_k$ is of finite cardinal. So $v\in V_2\cap \GG_{(\underset{1\leq k\leq n}{\bigcup} a_k)}\subseteq V_2$.

Therefore $V_1\subseteq  V_2$.
\end{proof}

A direct consequence of Lemma.(\ref{vect-int}) is that if for any finite $a\in \Pa(I)$, 

\begin{equation}
V_1\cap \GG_a = V_2\cap \GG_a.
\end{equation}

Then $V_1= V_2$.\\

\textit{Proof of the Theorem.(\ref{infi-intersect})}. Let $(\A_j)_{j\in J}$ be a family of elements of $\U(\Pa(I))$. Let $a\subseteq I$ of finite cardinal. \\

$\underset{j\in J}{\bigcap}  \GG_{\A_j} \cap \GG_a= \underset{j\in J}{\bigcap} \left(\GG_{\A_j} \cap \GG_a\right)$.\\

But, $\GG_{\A_j} \cap \GG_a =\GG_{\A_j \cap \widehat{\{a\}}}$. And $\{\GG_{\A_j \cap \widehat{\{a\}}}:\quad j\in J \}$ is finite, so $\underset{j\in J}{\bigcap} \left(\GG_{\A_j} \cap \GG_a\right)$ can be rewritten as a finite intersection and by Corollary (\ref{inf-intersect}),

$$\underset{j\in J}{\bigcap} \left(\GG_{\A_j} \cap \GG_a\right)=\GG_{\underset{j\in J}{\bigcap} (\A_j \cap \widehat{\{a\}})}\subset\GG_{\underset{j\in J}{\bigcap} \A_j} \cap \GG_a\quad .$$

By Lemma.(\ref{vect-int}), 

$$\bigcap \limits_{j\in J} \GG_{\A_j}\subseteq \GG_{\bigcap \limits_{j\in J}\A_j}\quad .$$

The other inclusion is always true (Remark (\ref{rem1})) as for any $i\in J$, $\bigcap \limits_{j\in J}\A_j\subseteq \A_i$.
\qed

\begin{rem} This proposition can also be stated in terms of the $\g_\A$ by taking the exponential: 

\begin{equation}
\bigcap \limits_{j\in J} \g_{\A_j}=\g_{\bigcap \limits_{j\in J}\A_j}.
\end{equation}

\end{rem}

\section{Applications} \label{chapitre-1:section-applications}

\subsection{Minimal factorisation}

In \citep{Yeung} a proof of the existence of a minimum factorisation is given, based on the existence of a decomposition into interaction subspaces, when $E$ is finite and $I$ finite. Let us recall that in a poset $\A$, $a\in \A$ is said to be a minimum if any $b\in \A$ is such that $a\leq b$. Let us give a proof of this result using Theorem \ref{infi-intersect} so without assuming that $I$ nor $E$ are finite.

\begin{cor}\label{Minfac2}

Let $I$ be any set and $E=\prod_{i\in I}E_i$ be the product of any collection of sets; for all $f\in \g$ let us call $\mathcal{F}(f)=\{ \g_\A | \quad f\in \g_\A\}$. $\mathcal{F}(f)$ admits a minimum and we say that $f$ admits a minimum decomposition.

\end{cor}

\begin{proof}
Let $\bA(f)=\{\A\in \U(\Pa(I)) | \quad f\in \g_\A\}$. From Theorem \ref{infi-intersect}, one has that,

$$\bigcap \limits_{\A\in \bA(f)} \g_{\A}=\g_{\bigcap \limits_{\A \in \bA(f)}\A}.$$

Any $\text{K} \in \mathcal{F}(f)$ contains $ \bigcap\limits_{\A\in \bA(f)} \g_{\A}$, therefore $\g_{\bigcap \limits_{\A \in \bA(f)}\A}$ is the minimum of $\mathcal{F}(f)$.

\end{proof}

\subsection{Markov properties and Hammersley-Clifford}\label{chapitre-1-section-4}
Let us consider four random variables $W,X,Y,Z$ taking values respectively in $E_0$, $E_1$, $E_2$, $E_3$  \textbf{finite sets}, with strictly positive joint law. Let us recall the law of $X$ conditionally to $Y$,

\begin{equation}
\forall (x,y)\in E_1\times E_2,\quad \p_{X|Y}(x,y)= \frac{\p_{X,Y}(x,y)}{\p_{Y}(y)}
\end{equation}

Conditional independence is usually defined as follows, 
\begin{equation}\label{condindep}
 X\independent Y |Z \quad \iff \quad \forall (x,y,z)\in E_1\times E_2 \times E_3, \quad \p_{(X,Y)|Z}(x,y,z)=\p_{X|Z}(x,z)\p_{Y|Z}(y,z)
\end{equation}

Let us pose $I=\{0,1,2,3\}$ we identify $\underset{i\in I}{\Pi} E_i$ with $E_0\times E_1\times E_2\times E_3$ by the following $x\mapsto(x(0),x(1),x(2),x(3))$ and then $\g_\A$ to sets in $\R^{E_0\times E_1\times E_2\times E_3}_{>0}$. Let $a=\{1,3\}$, $b= \{2,3\}$ and $\A=\{a,b\}$,

\begin{equation}
X\independent Y |Z \quad  \iff \quad \p_{X,Y,Z}\in  \g_\A.
\end{equation}

\begin{prop}[Bayesian or Graphoid intersection property]\label{Graphoid_interesect}
\begin{equation}
(X\independent Y |(Z,W))\wedge (X\independent W)|(Z,Y) \implies X\independent (Y,W)|Z.
\end{equation}
\end{prop}

\begin{proof}
Let $a=\{0,1,3\}$, $b=\{0,2,3\}$, $c=\{1,2,3\}$, $d=\{1,3\}$ and $\A=\{a,b\}$, $\B=\{b,c\}$, $\C=\{b,d\}$. $\A\sqcap \B\equiv \{a\cap c, b\}=\{d,b\}$ so $\g_\A \cap \g_\B \subseteq \g_\C$.
\end{proof}

\begin{cor}(Hammersley-Clifford)\label{HC}\\
 
 Let $G=(I,A)$ be a finite graph. For all strictly positive probability law, $P_X$, on a finite $E$,

\begin{equation}
P_X\in P(G) \iff P_X\in L(G)\iff P_X \in  \GG_\mathcal{C}
\end{equation}

\end{cor}

For any pair $(i,j)$ of elements of $I$ and for all probability law $P$,

\begin{equation}
X_i\independent X_j |X_{I\setminus \{i,j\}} \iff P_X\in \GG_{[i,j]}
\end{equation}

Let $\mathcal{A}_{P}=\underset{(i,j): \text{ } i\notin  \partial j}{\bigsqcap}[i,j]$. Similarly, for all $i\in I$,

\begin{equation}
X_i \independent X_{I \setminus (i\cup \partial i) }|X_{\partial i}\iff 
\mathbb{P}_X\in \GG_{[i]}
\end{equation}

Let $\mathcal{A}_{L}= \underset{i}{\bigsqcap}[i]$. The following lemma is the version of the Hammersley-Clifford on graphs that we will then translate to graphical models by applying $\Psi$. \\

\begin{lem}\label{graph_clique}

\begin{equation}
\hat{\mathcal{A}}_{L}=\hat{\mathcal{A}}_{P}= \mathcal{C}
\end{equation}
\end{lem}

\begin{proof}

Firstly, $\hat{\mathcal{A}}_{L}=\bigcap\limits_{(k,l): \text{ } k\notin  \partial l}\widehat{[k,l]}$. Let $a\in\hat{\mathcal{A}}_{L}$ and assume that $a$ is not a clique. So there is $i,j\in a $ such that $i\notin \partial j$. But $a\in \widehat{[i,j]}$, so $a\subseteq i\cup (I\setminus \{i,j\})$ or $a\subseteq j\cup (I\setminus \{i,j\})$. It is not possible as any of these two sets separate $i$ and $j$. So $a$ must be a clique. In other words, $\{i,j\}\subseteq a$ but $\{i,j\}\not\subseteq i\cup (I\setminus \{i,j\})$ and $\{i,j\}\not\subseteq j\cup (I\setminus \{i,j\})$ ($\{i,j\}\notin  \widehat{[i,j]}$). So if $a$ is not a clique of $G$, $a\not \in \hat{\mathcal{A}}_{L}$.\\

Suppose $a$ is a clique of $G$. Let $i,j\in I $ such that $i\notin \partial j$.  $i\cup (I\setminus \{i,j\})$ and $ j\cup (I\setminus \{i,j\})$ separate $i,j$. So a clique most be in only one of the two sets. To be more formal, for any subset $a$ of $I$, there is $b\subseteq I\setminus \{i,j\}$, such that $a=b$ or $a=b\cup i$ or $a=b\cup j$ or $a=b\cup \{i,j\}$. As $a$ is a clique $\{i,j\}\not\subseteq a$. So there is $b\subseteq I\setminus \{i,j\}$, such that $a=b$ or $a=b\cup i$ or $a=b\cup j$. Which is equivalent to saying that $a\in \hat{[i,j]}$.\\

So we proved that,

$$\hat{\mathcal{A}}_{P}= \mathcal{C}.$$

For the local case, $\hat{\mathcal{A}}_{L}$, one has to remark that $a$ is a clique of $G$ if and only if for all $i\in a$, $a\subseteq\{i, \partial i\}$ (for exemple see slide $6$ \citep{Langseth}).

\end{proof}

\textit{Proof of Corollary \ref{HC}.}
Let us remark that $P_X \in  P(G)$ if and only if $P_X \in \bigcap\limits_{(i,j): \text{ } i\notin  \partial j} \GG_{[i,j]}$ and similarly  $P_X \in  L(G)$ if and only if $P_X \in \bigcap \limits_{i\in I} \GG_{[i]}$.\\

As $P_X$ is stricly positive, by Corollary \ref{Morphisminter},

\begin{equation}
P_X \in  P(G) \iff P_X \in \GG_{\mathcal{A}_{P}}\iff P_X \in \GG_{\mathcal{C}}.
\end{equation}

\begin{equation}
P_X \in  L(G) \iff P_X \in G_{\mathcal{A}_{L}}\iff P_X \in \GG_{\mathcal{C}}.
\end{equation}
\qed

Similarly, when $G=(I,D)$ is any graph and $(E_i)_{i\in I}$ is any collection of sets, Lemma \ref{graph_clique} still holds and one has the following result which extends the Hammersley-Clifford theorem.\\

\begin{cor}\label{HC-infty}
 
\begin{equation}
\bigcap\limits_{(i,j): i\not\in \partial j}\g_{\widehat{[i,j]}}=\bigcap\limits_{i \in I}\g_{\widehat{[i]}}= \g_{\mathcal{C}}.
\end{equation}

\end{cor}

\section*{Acknowledgement}

This work resulted from research supported by the University of Paris. I am very grateful to Daniel Bennequin for our numerous discussions.

\bibliographystyle{plainnat}
\bibliography{bintro}

\begin{thebibliography}{9}
\providecommand{\natexlab}[1]{#1}
\providecommand{\url}[1]{\texttt{#1}}
\expandafter\ifx\csname urlstyle\endcsname\relax
  \providecommand{\doi}[1]{doi: #1}\else
  \providecommand{\doi}{doi: \begingroup \urlstyle{rm}\Url}\fi

\bibitem[Bourbaki(1939)]{Bourbaki}
Nicolas Bourbaki.
\newblock \emph{Th\'eorie des ensembles}.
\newblock Springer, 1939.

\bibitem[Chan and Yeung(2011)]{Yeung}
T.~H. Chan and Raymond~W. Yeung.
\newblock Probabilistic inference using function factorization and divergence
  minimization.
\newblock In M.~Dehmer et~al., editor, \emph{Towards an Information Theory of
  Complex Networks: Statistical Methods and Applications}, chapter~3, pages
  47--74. Springer, 2011.

\bibitem[Dawid(2001)]{Dawid}
Alexander~Philip Dawid.
\newblock Separoids: A mathematical framework for conditional independence and
  irrelevance.
\newblock \emph{Annals of Mathematics and Artificial Intelligence 3}, 2001.

\bibitem[Langseth()]{Langseth}
Helge Langseth.
\newblock The {H}ammersley-{C}lifford {T}heorem and its impact on modern
  statistics.

\bibitem[Lauritzen(1996)]{Lauritzen}
Steffen~L. Lauritzen.
\newblock \emph{Graphical Models}.
\newblock Oxford Science Publications, 1996.

\bibitem[Pearl(1988)]{Pearl1988}
Judea Pearl.
\newblock \emph{Probabilistic Reasoning in Intelligent Systems: Networks of
  Plausible Inference}.
\newblock Morgan Kaufmann Publishers, 1988.

\bibitem[Sergeant-Perthuis(2019)]{GS2}
Gr\'egoire Sergeant-Perthuis.
\newblock Intersection property and interaction decomposition.
\newblock arXiv:1904.09017v1, 2019.

\bibitem[Speed(1979)]{Speed}
Terry~P. Speed.
\newblock A note on nearest-neighbour gibbs and markov probabilities.
\newblock \emph{Sankhy\=a: The Indian Journal of Statistics, Series A}, 1979.

\bibitem[Yeung(2002)]{FirstC}
Raymond~W. Yeung.
\newblock \emph{A First Course in Information Theory}.
\newblock Springer, 2002.

\end{thebibliography}

\end{document}